\documentclass[11pt]{amsart}
\usepackage{amssymb,latexsym,amsmath,amsthm}
\usepackage[colorlinks=true,linkcolor=red,urlcolor=blue,citecolor=blue]{hyperref}
\usepackage{enumerate}
\usepackage[latin1]{inputenc}

\newtheorem{prop}{Proposition}[section]
\newtheorem{thm}[prop]{Theorem}
\newtheorem{cor}[prop]{Corollary}
\newtheorem{lem}[prop]{Lemma}

\theoremstyle{definition}
\newtheorem{defn}[prop]{Definition}
\newtheorem{rmk}[prop]{Remark}

\numberwithin{equation}{section}

\renewcommand{\Re}{\mathbb R}
\renewcommand{\epsilon}{\varepsilon}
\newcommand{\Ren}{\Re^n}

\renewcommand{\phi}{\varphi}

\renewcommand{\epsilon}{\varepsilon}
\newcommand{\mcuerpos}{K_1,\dots ,K_m}
\newcommand{\sumam}[1]{K_1+_{#1}\cdots+_{#1} K_m}

\DeclareMathOperator{\conv}{conv}

\DeclareMathOperator{\interior}{int}

\allowdisplaybreaks

\begin{document}

\title[B-M and Zhang inequalities for Convolution Bodies]{Brunn-Minkowski and Zhang inequalities for Convolution Bodies}

 \author[B-M and Zhang inequalities for Convolution Bodies]{David Alonso-Guti\'errez${}^*$ \and C. Hugo~Jim\'{e}nez${}^\dag{}^\ddag$\and Rafael~Villa${}^\dag$}

\address{C. Hugo Jim\'enez\and Rafael Villa, Dept. de An\'alisis Matem\'atico, Facultad de Matem\'aticas, Universidad de Sevilla, PO Box 1160, 41080 Sevilla, Spain}
\email{carloshugo@us.es,villa@us.es}
\address{David Alonso-Guti\'errez, Dept. of Math. and Stats., 632 Central Academic Building, University of Alberta, Edmonton, Ab, Canada T6G 2G1}
\email{alonsogu@ualberta.ca}

 \thanks{${}^\dag$ Authors partially supported by the Spanish Ministerio de Econom\'ia y Competitividad grant MTM2012-30748 and by the Junta de Andaluc\'{i}a, grant P08-FQM-03543.}
 \thanks{${}^\ddag$ Author partially supported by CONACyT}
\thanks{${}^*$ Spanish grant MTM2010-16679 and "Programa de Ayudas a Grupos de Excelencia de la región de Murcia", Fundación Séneca, 04540/GERM/06}


\maketitle

\begin{abstract}
A quantitative version of Minkowski sum, extending the definition of $\theta$-convolution of convex bodies, is studied to obtain extensions of the Brunn-Minkowski and Zhang inequalities, as well as, other interesting properties on Convex Geometry involving convolution bodies or polar projection bodies. The extension of this new version to more than two sets is also given.

\end{abstract}

\section{Introduction and motivation}

The Minkowski sum of two sets $A,B\subseteq\Re^n$ is defined as the set
$$
A+B=\{x\in\Re^n\,:\,A\cap(x-B)\neq\emptyset\}.
$$
The essential sum in terms of measure is defined as
$$
A+_eB=\{x\in\Re^n: |A\cap(x-B)|>0\},
$$
for $A,B\subseteq\Re^n$ measurable sets, when $|\cdot|$ denotes the Lebesgue measure in $\Re^n$. A quantitative version of this definition, involving the proportional measure of the intersections, gives the following subset of $A+B$
$$
A+_\theta B=\{x\in A+B: |A\cap(x-B)|\geq\theta M(A,B)\}
$$
 for $\theta\in[0,1]$, whenever $M(A,B):=\displaystyle{\sup_{x\in A+B}|A\cap (x-B)|}$ is finite. This set is called the \emph{$\theta$-convolution set} of $A$ and $B$. Note that $A+_0B$ is the usual Minkowski sum $A+B$. This set is studied  for symmetric convex bodies%
\footnote{A convex body is a compact convex subset of $\Ren$ with non-empty interior.}
in \cite{Kie,Sch1,M-R-S,Sch2} and \cite{Tso}, where the term {\it convolution body} is first introduced. However, our notation differs from the one used there, in order to emphasize the connection with the standard Minkowski sum. Properties of $\theta$-convolution bodies are given in Section \ref{propertiesandresults}.


Recently, there is an increasing interest in finding extensions of the classical
integral geometry of the motion group in Euclidean spaces to the group of translations
(see \cite{Schn2} and the references therein), motivated by possible
applications to the stochastic geometry of homogeneous random geometric structures.


Our purpose is to find volume estimates, from above and below, of the $\theta$-convolution of two sets. In what follows we will motivate our interest in studying the volume of this family of sets.

The celebrated Brunn-Minkowski inequality states that
\begin{equation*}
|K+L|^{\frac{1}{n}}\geq |K|^{\frac{1}{n}}+|L|^{\frac{1}{n}}
\end{equation*}
for two convex bodies $K,L\subset\Ren$, with equality if and only if $K$ and $L$ are homothetic. Recall that $K$ and $L$ are called \textit{homothetic} if $L=z+\lambda K$ for some $z\in\Ren$ and $\lambda>0$. Two classical references in this topic are \cite{Bo-Fe} and \cite{Schn1}.

Brunn-Minkowski inequility has been widely applied to solve a large number of problems involving geometrical quantities such as volume, surface area, and mean width. In the last thirty years the Brunn-Minkowski inequality has become an essential analytical tool to develop the so-called Local Theory of Normed Spaces and Convex Geometric Analysis \cite{Gar2,Schn1,Bo-Fe,M-S}. Extensions of this inequality for non convex sets, even for non-measurable sets, have been also studied. We  refer only to \cite{Had} for details and references.

In Section \ref{B-M_ineq_for_theta} a generalization of the Brunn-Minkowski inequality is studied. Even though extensive work with this inequality as backbone has emerged both within the class of convex bodies \cite{Fir,Lut1,Lut2} and under other settings \cite{Bor,Vit,Uhr,B-L}, we pursuit something closer in spirit to \cite{Bart1,Sz-V}. See \cite{Gar1} for a comprehensive survey on the Brunn-Minkowski inequality including extensions, applications and its relation to other analytical inequalities.\\ Namely, we pose the problem of finding the best function $\phi_n(\theta)$ such that
\begin{equation}\label{ineqcontheta}
|K+_\theta L|^{\frac{1}{n}}\geq \phi_n(\theta)^{\frac{1}{n}}(|K|^{\frac{1}{n}}+|L|^{\frac{1}{n}})
\end{equation}
for any convex bodies $K,L\subset\Ren$. It is proved that $\phi_n(\theta)=(1-\theta^{\frac{1}{n}})^n$ satisfies (\ref{ineqcontheta}). Some particular cases are also studied.

Following the work of Kiener \cite{Kie}, Schmuckensl\"{a}ger \cite{Sch1} proved that for any convex body $K$ of volume 1,
\begin{equation}\label{Sc}
(1-\theta)\Pi^*(K)\subseteq K+_\theta (-K)\subseteq\log\frac{1}{\theta}\Pi^*(K)
\end{equation}
where $\Pi^*(K)$ is the polar projection body of $K$, the unit ball of the norm $\Vert x\Vert_{\Pi^*(K)}=|x||P_{x^\bot}K|$. Here $P_{x^\bot}$ denotes the orthogonal projection on the hyperplane orthogonal to $x$.

These inclusions imply $\displaystyle{|K|\Pi^*(K)=\lim_{\theta\to 1^-}\frac{K+_\theta (-K)}{1-\theta}}$ in the Haussdorff metric for any convex body $K$.

In Section \ref{polarprojectionbodies} we modify the argument to improve the estimate (\ref{Sc}) (see Proposition \ref{Schmuckeslagerseveralbodies})
$$
(1-\theta)|K|\Pi^*(K)\subseteq K+_\theta (-K)\subseteq n(1-\theta^\frac{1}{n})|K|\Pi^*(K).
$$

The most famous inequality concerning the volume of the polar projection body of a convex body $K\subset\Re^n$ is Petty projection inequality,
$$
|K|^{n-1}|\Pi^*(K)|\leq\left(\frac{\omega_n}{\omega_{n-1}}\right)^n
$$
where $\omega_n$ denotes the volume of the $n$-dimensional Euclidean ball. The equality is attained provided $K$ is an ellipsoid (See \cite{Pet}). A different proof using convolutions can be found in \cite{Sch2}.

In \cite{Zha}, Zhang proved a reverse form of this inequality
\begin{equation}\label{inq:zhang}
|K|^{n-1}|\Pi^*(K)|\geq\frac{1}{n^n}{2n \choose n}
\end{equation}
for any convex body, with equality if and only if $K$ is a simplex. Zhang inequality can be written as
\begin{equation}\label{inq:zhang2}
\left|\lim_{\theta\to 1^-}\frac{K+_\theta (-K)}{1-\theta} \right|\geq\frac{1}{n^n}{2n \choose n}|K|.
\end{equation}

It is worth mentioning that Tsolomitis  studies in \cite{Tso} the behavior of limiting convolution bodies
\begin{equation}\label{limitcon}
\lim_{\theta\to 1^-}\frac{K+_\theta L}{(1-\theta)^\alpha}
\end{equation}
for symmetric convex bodies $K$ and $L$, and some exponent $\alpha$, giving some regularity conditions under which (for some specific $\alpha$) the limit (\ref{limitcon}) is  non-degenerated,  denoted by $C(K,L)$.

In \cite{R-S1}, Rogers and Shephard obtained the inequality
\begin{equation}\label{ineq:RoS1}
|K-K|\leq{2n \choose n}|K|
\end{equation}
for any convex body $K$, with equality if and only if $K$ is a simplex.
Throughout the proof it is showed that
\begin{equation}\label{eqn:RoS2}
K+_\theta(-K)\supseteq (1-\theta^{\frac{1}{n}})(K-K)
\end{equation}
with equality if and only if $K$ is a simplex. They also showed in \cite{Ro-Sh2} the extension for two different convex bodies
\begin{equation}\label{eqn:RoS3}
|K-L||K\cap L|\leq {2n \choose n}|K||L|.
\end{equation}

The last part of Section \ref{polarprojectionbodies} is devoted to generalize the inclusions stated in (\ref{Sc}) and Zhang inequality (\ref{inq:zhang}) for limiting convolutions of different convex bodies. This generalization is a consequence of Corollary \ref{increasingintheta}, from which (\ref{eqn:RoS3}) can be obtained (see Proposition \ref{rmk:RoS3}).


Given that some classical geometric inequalities are recovered and extended through the use of convolution bodies,
it would be natural to consider the 
extension of the convolution body of more than two bodies in order to generalize these inequalities.
In Section \ref{mbodies} we study such an extension obtaining similar inequalities when we consider more than two bodies. Surprisingly, it turns out when studying the equality cases, that these inequalities can only be sharp when convoluting two convex bodies and not when considering three or more bodies.


\section{Properties of the $\theta$-convolution of convex bodies}\label{propertiesandresults}

In this section we give some properties of the $\theta$-convolution of two convex bodies, from which the Brunn-Minkowski-type inequality for the $\theta$-convolution of convex bodies $|K+_\theta L|^{\frac{1}{n}}\geq (1-\theta^{\frac{1}{n}})\left(|K|^{\frac{1}{n}}+|L|^{\frac{1}{n}}\right)$ will follow. However, this bound is not sharp, as we will see below.

We now list some basic properties of $M(K,L)$ and the $\theta$-convolution in the following\\

\begin{prop}\label{prop:properties}
Let $K,L\subset\Re^n$ be compact sets, $\lambda\geq0$, $x\in\Re^n$  and $T\in GL_n(\Re)=\{T:\Re^n\rightarrow\Re^n \ :\ T \text{ is linear}\}$.

\begin{itemize}
\item[(a1)]$M(K,L)=M(L,K)$.
\item[(a2)]$M(x+K,L)=M(K,L)$.
\item[(a3)]$M(\lambda K,\lambda L)=\lambda^nM(K,L)$.
\item[(a4)]$M(TK,TL)=|\det T|M(K,L)$.
\item[(a5)]If $K=-L$ or $K,L$ are symmetric, then $M(K,L)=|K\cap (-L)|$.
\end{itemize}
Let  $\theta\in [0,1]$.
\begin{itemize}
\item[(b1)]$(\lambda K)+_{\theta}(\lambda L)=\lambda(K+_{\theta}L)$.
\item[(b2)]$K+_{\theta}L=L+_\theta K$.
\item[(b3)]$(x+K)+_{\theta}L=x+(K+_{\theta}L)$.
\item[(b4)]$TK+_{\theta}TL=T(K+_{\theta}L)$.
\end{itemize}
\end{prop}

A first question about this $\theta$-convolution is its convexity, provided that $K$ and $L$ are both convex. The affirmative answer is a consequence of the following result. In what follows, using $(b3)$ in Proposition \ref{prop:properties} above, we will assume without loss of generality, that \begin{equation}\label{Mzero}
M(K,L)=|K\cap(-L)|.
\end{equation}

\begin{prop}\label{incconvexity}
    Let $K,L\subset\Re^n$ be convex bodies satisfying (\ref{Mzero}). Then for every $\theta_1,\theta_2, \lambda_1,\lambda_2 \in[0,1]$ such that $\lambda_1+\lambda_2\leq 1$ we have that
    \begin{equation}
    \label{inclusion1:convexity}
    \lambda_1 (K+_{\theta_1}L)+ \lambda_2(K+_{\theta_2}L)\subseteq K+_{\theta} L,
    \end{equation}
    where $1-\theta^{\frac{1}{n}}=\lambda_1(1-\theta_1^{\frac{1}{n}})+\lambda_2(1-\theta_2^{\frac{1}{n}}).$
    \end{prop}

    \begin{proof}
    Let $x_1\in K+_{\theta_1}L$ and $x_2\in K+_{\theta_2}L$. From the general inclusion
$$
K\cap(\lambda_0A_0+\lambda_1A_1+\lambda_2A_2)
\supset
\lambda_0K\cap A_0+\lambda_1K\cap A_1+\lambda_2K\cap A_2
$$
where $K$ is convex and $\lambda_0+\lambda_1+\lambda_2=1$, and using the convexity of  $L$, we have
    \begin{equation}\label{inclusion0:convexity}
    K\cap (\lambda_1 x_1+\lambda_2 x_2-L)\supseteq(1-\lambda_1-\lambda_2)[K\cap(-L)]+\lambda_1[K\cap(x_1-L)]+\lambda_2[K\cap (x_2-L)].
    \end{equation}
    Taking volumes, using the classical Brunn-Minkowski inequality and the fact that $x_i\in K+_{\theta_i}L$ we have
    \begin{equation}\label{BM:convexity}
    |K\cap(\lambda_1x_1+\lambda_2x_2-L)|\geq[1-\lambda_1(1- \theta_1^{\frac{1}{n}})-\lambda_2(1-\theta_2^{\frac{1}{n}})]^nM(K,L),
    \end{equation}
    which proves that $\lambda_1x_1+\lambda_2x_2\in K+_\theta L$ for $\theta=[1-\lambda_1(1- \theta_1^{\frac{1}{n}})-\lambda_2(1-\theta_2^{\frac{1}{n}})]^n$.
       \end{proof}

Taking $\theta_1=\theta_2$ and $\lambda_2=1-\lambda_1$ we have

\begin{cor}
Let $K,L\subset\Re^n$ be convex bodies and $\theta\in[0,1]$. Then $K+_\theta L$ is  convex.
\end{cor}

The following result on $K+_1L$ will be used later on, and it is a consequence of Proposition \ref{incconvexity}.

\begin{cor}\label{convolution1:equality}
Let $K,L\subset\Re^n$ be convex bodies. For any $x\in K+_1 L$, $K\cap (x-L)$ is a translation of $K\cap(-L)$.

\end{cor}

\begin{proof}
Taking $\theta_1=\theta_2=1$ in Proposition \ref{incconvexity}, we get $\theta=1$, and then Brunn-Minkowski inequality (\ref{BM:convexity}), obtained from (\ref{inclusion0:convexity}), holds with equality. Then (see equality cases in Brunn-Minkowski inequality in \cite{Br}) the four sets involved in (\ref{inclusion0:convexity}),
$$
K\cap (\lambda_1x_1+\lambda_2x_2-L),
\hskip.5cm
K\cap(-L),
\hskip.5cm
K\cap(x_1-L),
\hskip.5cm
K\cap(x_2-L),
$$
are all homothetic. As they all have the same volume (equals to $M(K,L)$), homotheties are indeed translations.
\end{proof}

Taking $\theta_1=\theta_2$ leads us to the following

\begin{cor}\label{increasingintheta}
Let $K,L\subset\Re^n$ be convex bodies satisfying (\ref{Mzero}). Then, for every $0\leq\theta_0\leq \theta<1$ we have
$$
\frac{K+_{\theta_0} L}{1-{\theta_0}^\frac{1}{n}}
\subseteq
\frac{K+_{\theta} L}{1-\theta^\frac{1}{n}}.
$$
\end{cor}

\begin{proof}
Taking $\theta_1=\theta_2=\theta_0$ in the above proposition, for any $\lambda_1,\lambda_2\in [0,1]$ such that $\lambda_1+\lambda_2\leq 1$
$$
\lambda_1(K+_{\theta_0} L)+
\lambda_2(K+_{\theta_0} L)
\subseteq
(\lambda_1+\lambda_2)(K+_{\theta_0} L)
\subseteq
 K+_{\theta} L,
$$
with $1-\theta^{\frac{1}{n}}=(\lambda_1+\lambda_2)(1-{\theta_0}^\frac{1}{n})$. Since $\lambda_1+\lambda_2=\displaystyle{\frac{1-\theta^\frac{1}{n}}{1-\theta_0^\frac{1}{n}}}$,
$$
\frac{1-\theta^\frac{1}{n}}{1-\theta_0^\frac{1}{n}}(K+_{\theta_0} L)
\subseteq
K+_{\theta} L
$$
whenever $\lambda_1+\lambda_2\leq 1$, which means $0\leq\theta_0\leq \theta\leq 1$.
\end{proof}

The following extension of (\ref{eqn:RoS2}) from Rogers and Shephard's work  will be used to get a first Brunn-Minkowski-type inequality.

\begin{cor}\label{inclusion:BM}
Let $K,L\subset\Re^n$ be convex bodies. Then
\begin{equation}\label{conttheta}
\theta^\frac{1}{n}(K+_1L)+(1-\theta^\frac{1}{n})(K+L)\subseteq K+_\theta L.
\end{equation}
\end{cor}
\begin{proof}
Take $\theta_1=1$, $\theta_2=0$, $\lambda_1=\theta^\frac{1}{n},$ and $\lambda_2=(1-\theta^\frac{1}{n})$ in Proposition \ref{incconvexity} to obtain the desired result.
\end{proof}

Note that the condition $0\in K+_1L$ is equivalent to $M(K,L)=|K\cap (-L)|$, and that is verified under the assumptions of Proposition \ref{prop:properties} $(a5)$.

We will use the following description of the boundary of $K+_\theta L$.

\begin{lem}\label{boundary}
Let $K,L\subset\Re^n$ be convex bodies, and $\theta\in[0,1)$. Then
$$
\partial (K+_\theta L)
=
\{x\in K+L : |K\cap (x-L)|=\theta M(K,L)\}.
$$
In particular, for $\theta=0$,
$$
\partial (K+L)
=
\{x\in K+L : |K\cap (x-L)|=0\}.
$$
Consequently, for any $x\in K+L$, $x\not\in K+_1 L$, there exists a unique $\theta\in [0,1)$ such that $x\in\partial (K+_\theta L)$.
\end{lem}

\begin{proof}
Let $f:\Ren\to[0,+\infty)$ be given by $f(x)=|K\cap (x-L)|^{\frac{1}{n}}$. Using properties of the Lebesgue measure and Brunn-Minkowski inequality, it can be deduced that $f$ is continuous on $\Ren$ and concave on $K+L$. It is then clear that $K+_\theta L$ is  closed.

The first assertion 
is equivalent to $\interior(K+_\theta L)=\{x\in K+L : f(x)^n>\theta M(K,L)\}$. Since the right-hand set is open and it is contained in $K+_\theta L$, it remains to be shown the inclusion
$\interior(K+_\theta L)\subseteq\{x\in K+L : f(x)^n>\theta M(K,L)\}.$

Let us take any $x\in K+L$, with $f(x)^n=\theta M(K,L)$. It is left to show that $x\not\in\interior(K+_\theta L)$.

We will assume without loss of generality that (\ref{Mzero}) is satisfied. For any $\lambda\in(0,1)$, write $x=\lambda(\lambda^{-1} x)+(1-\lambda)\cdot 0$. If $\lambda^{-1} x\in K+L$, using the concavity of $f$ on $K+L$,
$$
(\theta M(K,L))^{\frac{1}{n}}
=
f(x)
\ge
\lambda f(\lambda^{-1} x)+(1-\lambda^{-1})f(0)
=
\lambda f(\lambda^{-1} x)+(1-\lambda^{-1})M(K,L)^{\frac{1}{n}}
$$
and therefore
$$
f(\lambda^{-1} x)^n
\le
\frac{\theta^{\frac{1}{n}} - (1-\lambda^{-1})}{\lambda}M(K,L)^{\frac{1}{n}}<(\theta M(K,L))^{\frac{1}{n}}.
$$

Then $\lambda^{-1} x\not\in K+_\theta L$ for any $\lambda<1$ (for $\lambda^{-1}x\notin K+L$ it is trivial). Hence $x\not\in\interior(K+_\theta L)$.

To obtain the last assertion, just take $\theta=\displaystyle\frac{|K\cap (x-L)|}{M(K,L)}\in[0,1)$.
\end{proof}

Equality cases in Proposition \ref{incconvexity} and Corollary \ref{inclusion:BM} are stated in the following result.

\begin{prop}\label{equalityconvexity1}
Let $K,L\subset\Re^n$ be convex bodies. The following conditions are equivalent.
\begin{enumerate}[(i)]
\item For every $\theta_1,\theta_2, \lambda_1,\lambda_2 \in[0,1]$ such that $\lambda_1+\lambda_2= 1$, we have
$$
\lambda_1 (K+_{\theta_1}L)+ \lambda_2(K+_{\theta_2}L)
=
K+_{\theta} L,
$$
where $\theta^\frac{1}{n}=\lambda_1\theta_1^\frac{1}{n}+\lambda_2\theta_2^\frac{1}{n}.$

\item For every $\theta\in[0,1]$,
$\displaystyle
{\theta}^\frac{1}{n}(K+_1 L)+
({1-{\theta}^\frac{1}{n}})(K+L)
=
{K+_{\theta} L}.
$
\item $K$ and $-L$ are homothetic $n$-simplices.
\end{enumerate}

\end{prop}

For the proof, we will use the following result from \cite{soltan} (we  state it here in the form it will be used in this paper).

\begin{thm}\label{soltanth} (Soltan)
Let $K, L\subset\Re^n$ be convex bodies. The following conditions are
equivalent:
\begin{enumerate}[(i)]
\item $K$ and $L$ are homothetic $n$-simplices.
\item The $n$-dimensional intersections $K \cap(z - L)$, $z \in\Ren$,
are all homothetic to $K\cap(-L)$.
\end{enumerate}
\end{thm}

\begin{proof}[Proof of Proposition \ref{equalityconvexity1}] Using a translation we may assume that (\ref{Mzero}) is satisfied.

Conditions (i) and (ii) are equivalent. Indeed, (i) trivially implies (ii).

Suppose (ii) holds. Then, using the equality for $\theta_1$, $\theta_2$ and $\theta$ successively, and the convexity of the convolution sets,
\begin{align*}
\lambda_1 &(K+_{\theta_1}L)+ \lambda_2(K+_{\theta_2}L)
\\
&=\lambda_1
\left({\theta_1}^\frac{1}{n}(K+_1 L)+
({1-{\theta_1}^\frac{1}{n}})(K+L)\right)
\\
&\hskip1cm+
\lambda_2
\left({\theta_2}^\frac{1}{n}(K+_1 L)+
({1-{\theta_2}^\frac{1}{n}})(K+L)\right)
\\
&=
(\lambda_1\theta_1^\frac{1}{n}+\lambda_2\theta_2^\frac{1}{n})
(K+_1L)
+
(\lambda_1(1-\theta_1^\frac{1}{n})+\lambda_2(1-\theta_2^\frac{1}{n}))
(K+L)
\\
&=K+_\theta L,
\end{align*}
and (i) follows.

Suppose (ii) holds, and take $x\in K+L$, $x\not\in\partial(K+L)$. If $x\in K+_1 L$, Corollary \ref{convolution1:equality} shows that $K\cap (x-L)$ is  a translation of (so homothetic to) $K\cap(-L)$.

If $x\not\in K+_1 L$, by Lemma \ref{boundary}, there is a $\theta\in(0,1)$ such that $x\in\partial(K+_\theta L)$. Using (ii), $x={\theta}^\frac{1}{n}x_1+
({1-{\theta}^\frac{1}{n}})x_2$, for some $x_1\in K+_1 L$, $x_2\in K+L$.

Now, inclusion (\ref{inclusion0:convexity}) and inequality (\ref{BM:convexity}) are both equalities (since $|K\cap(x-L)|=\theta M(K,L)$), and then $K\cap(x-L)$, $K\cap(x_1-L)$ and $K\cap(x_2-L)$ are all homothetic (see equality cases in Brunn-Minkowski inequality in \cite{Br}). Since $x_1\in K+_1 L$, they are all homothetic to $K\cap(-L)$.

Then, all the $n$-dimensional intersections $K\cap(x-L)$ are homothetic to the same body $K\cap(-L)$.
Now it follows from Theorem \ref{soltanth} that $K$, $-L$ and $K\cap(-L)$ are homothetic simplices.

Now suppose condition (iii) holds. After an affine transformation we may assume that
$$
K=\left\{t\in\Ren : t_i\ge0,\,\sum_{i=1}^n t_i\le1\right\}
$$
and $-L=\lambda K$ with $0<\lambda\le1$.
For any $x\in\Ren $,
$$
K\cap(x-L)=\left\{t\in\Ren : t_j\ge x_j^+,\,\sum_{i=1}^n t_i\le\min\{1,\lambda+\sum_{i=1}^n x_i\}\right\},
$$
where $r^+=\max\{r,0\}$. Then, $K\cap(x-L)=z(x)+\lambda(x)K$, where
$
z_j(x)=x_j^+,$ and
$$
\lambda(x)=\min\{1,\lambda+\sum_{i=1}^n x_i\}-\sum_{i=1}^n x_i^+
$$
whenever $\lambda(x)\ge0$. For those $x\in\Ren$, we have $|K\cap(x-L)|=\lambda(x)^n|K|$.
It is easy to see that $\lambda(x)\le \lambda$, and equality holds if and only if $x_j\ge0$ for all $j$ and $\lambda+\sum_{i=1}^n x_i\le1$. Then $M(K,L)=\lambda^n|K|$, and $K+_1L=(1-\lambda)K$.

Using absolute values, $\lambda(x)$ can be rewritten as
$$
\lambda(x)=
\frac{1}{2}\left(1+\lambda-\left|1-\lambda-\sum_{i=1}^n x_i\right|-\sum_{i=1}^n|x_i|\right).
$$
Then
$$
K+_\theta L
=
\left\{
x\in\Ren:
\left|1-\lambda-\sum_{i=1}^n x_i\right|+\sum_{i=1}^n|x_i|\le 1+\lambda(1-2\theta^{\frac{1}{n}})
\right\}.
$$
In particular, letting $\theta=0$, we obtain an expression for $K+L$. In order to prove (ii), it is enough to prove the inclusion $\displaystyle {K+_{\theta} L} \subseteq {\theta}^\frac{1}{n}(K+_1 L)+ ({1-{\theta}^\frac{1}{n}})(K+L)$
for every $\theta\in(0,1)$. Any $x\in {K+_{\theta} L}$ satisfies
\begin{equation}\label{ineqx}
\left|1-\lambda-\sum_{i=1}^n x_i\right|+\sum_{i=1}^n|x_i|
\le
\theta^{\frac{1}{n}}(1-\lambda)+(1-\theta^{\frac{1}{n}})(1+\lambda).
\end{equation}

Recall that $\displaystyle \left|1-\lambda-\sum_{i=1}^na_i\right|+\sum_{i=1}^n|a_i|=1-\lambda$ for any $a\in K+_1 L$. Consequently, if
\begin{equation}\label{decompx}
x=
{\theta}^\frac{1}{n}a+ ({1-{\theta}^\frac{1}{n}})b,
\end{equation}
with $a\in K+_1L$, then $b\in K+L$  provided that the left hand side in (\ref{ineqx})
$$
\left|
\theta^{\frac{1}{n}}(1-\lambda-\sum_{i=1}^n a_i)+(1-\theta^{\frac{1}{n}})(1-\lambda-\sum_{i=1}^n b_i)\right|
+
\sum_{i=1}^n|{\theta}^\frac{1}{n}a_i+ ({1-{\theta}^\frac{1}{n}})b_i|
$$
equals
$$
\theta^{\frac{1}{n}}
\left|
1-\lambda-\sum_{i=1}^n a_i
\right|
+
(1-\theta^{\frac{1}{n}})
\left|
1-\lambda-\sum_{i=1}^n b_i
\right|
+\theta^\frac{1}{n}
\sum_{i=1}^n|a_i|
+
(1-{\theta}^\frac{1}{n})
\sum_{i=1}^n|b_i|.
$$
Considering equality cases in triangle inequality, this happens provided that $1-\lambda-\sum_{i=1}^n a_i$ has the same sign as $1-\lambda-\sum_{i=1}^n b_i$, and for any $j$, $a_j$ has the same sign as $b_j$ (here $r,s$ have the same sign iff $r\cdot s\ge0$).

 If $\sum_{i=1}^n x_i^+\le1-\lambda$, it is enough to consider  $a\in K+_1 L$ in (\ref{decompx}), so that $a_j=x_j^+$. Then $b_j=({1-{\theta}^\frac{1}{n}})^{-1}x_j$ if $x_j<0$ and $b_j=x_j$ if $x_j\ge0$. In any case, $b_j$ has the same sign as $a_j$, and therefore $b\in K+L$.

If $\sum_{i=1}^n x_i^+>1-\lambda$, it is enough to take $a\in K+_1 L$ in (\ref{decompx}), so that $a_j=(1-\lambda)\frac{x_j^+}{\sum_{i=1}^n x_i^+}$. Then $b\in K+L$ provided that $b_j\ge0$ for those $j$ so that $x_j>0$. But in that case
$$
x_j={\theta}^\frac{1}{n}\frac{(1-\lambda)}{\sum_{i=1}^nx_i^+}x_j
+
({1-{\theta}^\frac{1}{n}})b_j
$$
Then $b_j\ge0$ provided that
$$
{\theta}^\frac{1}{n}\frac{(1-\lambda)}{\sum_{i=1}^n x_i^+}
\le1,
$$
which  is assumed to be true. In any case, we can find a decomposition as (\ref{decompx}), which proves (ii).
\end{proof}

Now we can deduce equality cases in Corollary \ref{increasingintheta}.

\begin{prop}\label{equalityconvexity2}
Let $K,L\subset\Re^n$ be convex bodies satisfying (\ref{Mzero}). The following conditions are equivalent.
\begin{enumerate}[(i)]
\item For any $0\le\theta_0\le\theta\le1$,
$\displaystyle
(1-\theta^\frac{1}{n})(K+_{\theta_0}L)
=
(1-\theta_0^\frac{1}{n})(K+_{\theta}L).
$

\item For every $\theta\in[0,1]$,
$\displaystyle
({1-{\theta}^\frac{1}{n}})(K+L)
=
{K+_{\theta} L}.
$
\item $K=-L$ is an $n$-simplex.
\end{enumerate}

\end{prop}

\begin{proof}
Again, conditions (i) and (ii) are easily seen to be equivalent.

Now, suppose (ii) holds; then $K+_1L=\{0\}$, and consequently (ii) in Proposition \ref{equalityconvexity1} holds. Then $K$ and $-L$ are homothetic $n$-simplices.

As in the proof of Proposition \ref{equalityconvexity1}, we have $K+_1L=(1-\lambda)K$, which implies $\lambda=1$, since $K+_1L=\{0\}$. Consequently, $K=-L$.


Condition (iii) implies (ii) in Proposition \ref{equalityconvexity1}, and since $K+_1L=\{0\}$, we get (ii).
\end{proof}

\section{Brunn-Minkowski type inequality for $\theta$-convolution bodies}\label{B-M_ineq_for_theta}

From the previous study on convolution of two sets, the following natural question arises: what kind of Brunn-Minkowski-type inequality for $\theta$-convolutions
\begin{equation}\label{eq:BMmod}
|K+_\theta L|^{\frac{1}{n}}\geq\varphi_n(\theta)^{\frac{1}{n}}(|K|^{\frac{1}{n}}+|L|^{\frac{1}{n}}).
\end{equation}
does it hold?

 As in the classical case, the homogeneity allows one to formulate the inequality in different equivalent forms.

\begin{prop}
The following statements are all equivalent:
\begin{enumerate}[(i)]
 \item For $K,L$ measurable sets in $\Re^n$
 $$
  |K+_\theta L|^{\frac{1}{n}}\geq\varphi_n(\theta)^{\frac{1}{n}}(|K|^{\frac{1}{n}}+|L|^{\frac{1}{n}}).
 $$
 \item For $K,L$ measurable sets in $\Re^n$ and $0<\lambda<1$
$$
|\lambda K+_\theta(1-\lambda) L|^{\frac{1}{n}}\geq\varphi_n(\theta)^{\frac{1}{n}}(\lambda|K|^{\frac{1}{n}}+(1-\lambda)|L|^{\frac{1}{n}}).
$$
 \item For $K,L$ measurable sets in $\Re^n$ and $0<\lambda<1$
 $$
 |\lambda K+_\theta(1-\lambda) L|\geq\varphi_n(\theta)(|K|^{\lambda}\cdot|L|^{1-\lambda}).
 $$
 \item For $K,L$ measurable sets in $\Re^n$ and $0<\lambda<1$
 $$
 |\lambda K+_\theta(1-\lambda) L|\geq\varphi_n(\theta)\min\{|K|,|L|\}.
 $$
 \item For $K,L$ measurable sets in $\Re^n$ such that $|K|=|L|=1$ and $0<\lambda<1$ $$|\lambda K+_\theta(1-\lambda) L|\geq \varphi_n(\theta).$$
\end{enumerate}
\end{prop}

\begin{proof}
$(i)\rightarrow(ii)$ and $(iii)\rightarrow(iv)\rightarrow(v)$ are immediate. The proof of $(ii)\rightarrow(iii)$ is obtained by taking logarithm and using its concavity.

Finally, apply ({\it{v}}) with $\overline{K}=|K|^{-\frac{1}{n}}K$, $\overline{L}=|L|^{-\frac{1}{n}}L$ and $\lambda=\frac{s|K|^{\frac{1}{n}}}{s|K|^{\frac{1}{n}}+t|L|^{\frac{1}{n}}}$,
and use the homogeneity of the convolution (Proposition \ref{prop:properties} (b1)) to get (\it{i}).
\end{proof}
A first inequality in this direction for convex bodies is obtained from Corollary \ref{inclusion:BM}.

\begin{cor}\label{BMtheta}
Let $K,\ L\subset\Re^n$ be convex bodies. Then
\begin{equation*}
|K+_\theta L|^{\frac{1}{n}}\geq (1-\theta^{\frac{1}{n}})(|K|^{\frac{1}{n}}+|L|^{\frac{1}{n}}).
\end{equation*}
Equivalently, $\varphi_n(\theta)\geq (1-\theta^{\frac{1}{n}})^n$ in (\ref{eq:BMmod}).
\end{cor}

\begin{proof}
Taking volumes in (\ref{conttheta})
\begin{equation}\label{cor:voltheta}
|K+_\theta L|^{\frac{1}{n}}\geq (1-\theta^{\frac{1}{n}})|K+L|^{\frac{1}{n}}
\end{equation}
and applying Brunn-Minkowski inequality
\begin{equation}\label{cor:teomasBM}
(1-\theta^{\frac{1}{n}})|K+L|^{\frac{1}{n}}\geq(1-\theta^{\frac{1}{n}})(|K|^{\frac{1}{n}}+|L|^{\frac{1}{n}})
\end{equation}
we obtain the desired result.
\end{proof}

In order to have equality in Corollary \ref{BMtheta}, we need to have equality in  (\ref{cor:voltheta}) and in Brunn-Minkowski inequality (\ref{cor:teomasBM}). However, by Proposition \ref{equalityconvexity2}, equality in (\ref{cor:voltheta}) holds if and only if  $K=-L$ is an $n$-dimensional simplex, and in that case there is not equality in Brunn-Minkowski inequality (unless $n=1$).
See examples at the end of the section for details.


The following result improves the inclusion
\begin{equation}\label{maincontent}
(1-\theta^\frac{1}{n})(K+L)\subseteq K+_\theta L
\end{equation} providing a new set between them. A good estimate for the volume of this new set would lead to a better estimate for $|K+_\theta L|$.

\begin{thm}\label{intermediateset}
Let $K,L\subset\Re^n$ be convex bodies such that $\displaystyle{M(K,L)=|K\cap(-L)|}$. Then for all $\theta\in[0,1]$,
\begin{align*}
  K+_\theta L &  \supseteq  \left\{a+b\ :\ a\in K,\ b\in L,\ \frac{|(1-||a||_K)K\cap(1-||b||_{L})(-L)|}{|K\cap(-L)|}\geq \theta \right\}&\\
   &\supseteq  (1-\theta^{\frac{1}{n}})(K+L).&
\end{align*}
\end{thm}
\begin{proof}
Let $x\in K+L$, then $x=a+b$ with $a\in K$ and $b\in L$. From the convexity of $K$
$$
(1-||a||_K)K+||a||_K\frac{a}{||a||_K}\subseteq K.
$$
Also, since $||-b||_{-L}=||b||_L$ and $L$ is convex
$$
(1-||b||_L)(-L)+||b||_L\frac{-b}{||b||_L}+x\subseteq x-L.
$$
Since $x-b=a$, we have $(1-||a||_K)K+a\subseteq K \text{ and }(1-||b||_L)(-L)+a\subseteq x-L$.
Thus,
$$
a+(1-||a||_K)K \cap(1-||b||_L)(-L)\subseteq K\cap (x-L)
$$
and then
$|K\cap (x-L)|\geq |(1-||a||_K)K \cap(1-||b||_L)(-L)|.$
Consequently,
$$
K+_\theta L   \supseteq  \{a+b\ :\ a\in K,\ b\in L,\ \frac{|(1-||a||_K)K\cap(1-||b||_{L})(-L)|}{|K\cap(-L)|}\geq \theta \}.
$$
This set trivially contains the set
$$
\{a+b\ :\ \inf\{ (1-||a||_K)^n,(1-||b||_L)^n \}\geq \theta\}=(1-\theta^{\frac{1}{n}})(K+L).\vspace{-0.9cm}$$\end{proof}

In order to get a more accurate idea of how good the bound in Corollary \ref{cor:voltheta} is, we estimate the quotient $\displaystyle{\frac{|K+_\theta L|^{\frac{1}{n}}}{|K|^{\frac{1}{n}}+|L|^{\frac{1}{n}}}}$ for some particular pairs of bodies.

\noindent\textbf{Examples:}
\begin{itemize}
\item[1)] For $K,\ L$ cubes whose sides are parallel to the coordinate hyperplanes, it is not hard to see that the quotient is minimized when $K=L=[-1/2,1/2]^n$, and its value equals
    $$\left[1-\theta\sum_{k=0}^{n-1}\frac{(-\log \theta)^k}{k!}\right]^{\frac{1}{n}}.$$
\item[2)] For $K=L$ the unit Euclidean ball, the quotient equals $R_n(\theta)$ given by the equality
    $$2\omega_{n-1}\int_{R_n(\theta)}^1\left(1-s^2\right)^{\frac{n-1}{2}}ds=\theta\omega_n$$
    where $\omega_n$ denotes the volume of the $n$-dimensional unit Euclidean ball.
\item[3)] As it was mentioned above, in \cite{R-S1} it was proved that, for $K=L$ the simplex, the quotient equals
    $$(1-\theta^{\frac{1}{n}})\frac{{2n\choose n}^{\frac{1}{n}}}{2}\sim 2(1-\theta^{\frac{1}{n}}).$$
\end{itemize}
Comparing these three cases, it seems that the minimum value for the quotient is attained in a different case depending on $\theta$. This fact makes difficult to find a family of bodies in which the minimum is attained.

\section{A connection with projection bodies and Zhang inequality}\label{polarprojectionbodies}

This section is devoted to generalize the inclusions (\ref{Sc}) and Zhang inequality for convolution of different convex bodies.

The following result generalizes the right hand side inclusion in (\ref{Sc}). We extend the ideas used in \cite{Sch1}
\begin{prop}\label{Schmuckeslagerseveralbodies}
Let $K,\ L\subset \Re^n$ be convex bodies satisfying (\ref{Mzero}). Then, for every $\theta\in(0,1)$
$$
K+_\theta L\subseteq\left\{x\in\Re^n\,:\,|x|\left|\frac{d^+}{dt}\left|K\cap \left(t\frac{x}{|x|}-L\right)\right|_{t=0}\right|\leq n(1-\theta^\frac{1}{n})M(K,L)\right\}.
$$
\end{prop}

\begin{proof}
The concavity of the function $x\mapsto|K\cap(x-L)|^\frac{1}{n}$ implies
\begin{eqnarray*}
|K\cap(\lambda x-L)|&\geq&\left((1-\lambda)M(K,L)^\frac{1}{n}+\lambda|K\cap(x-L)|^\frac{1}{n} \right)^n\cr
&=&M(K,L)\left[1+\lambda\left(\frac{|K\cap(x-L)|^\frac{1}{n}}{M(K,L)^\frac{1}{n}}-1 \right) \right]^n\cr
&\geq&M(K,L)\left[1+\lambda n\left(\frac{|K\cap(x-L)|^\frac{1}{n}}{M(K,L)^\frac{1}{n}}-1 \right) \right]
\end{eqnarray*}
for $\lambda\in [0,1]$ and $x\in K+L$. On the other hand,
\begin{eqnarray*}
|K\cap (\lambda x-L)|&=&M(K,L)+\int_0^{\lambda|x|}\frac{d^+}{dt}\left|K\cap \left(t\frac{x}{|x|}-L\right)\right|dt\\
&\leq&M(K,L)+\lambda|x|\max_{t\in[0,\lambda|x|]}\frac{d^+}{dt}\left|K\cap \left(t\frac{x}{|x|}-L\right)\right|
\end{eqnarray*}
again using the concavity of $x\mapsto|K\cap(x-L)|^\frac{1}{n}$. Comparing these two inequalities, and letting $\lambda\to0^+$, we obtain
$$
nM(K,L)\left(\frac{|K\cap(x-L)|^\frac{1}{n}}{M(K,L)^\frac{1}{n}}-1 \right)\leq|x|\frac{d^+}{dt}\left|K\cap \left(t\frac{x}{|x|}-L\right)\right|_{t=0}.
$$
Since the lateral derivative is non positive, we get the desired inclusion.
\end{proof}

\begin{rmk}
\item[1)]If $L=-K$, then the right-hand side set is exactly $n(1-\theta^\frac{1}{n})|K|\Pi^*K$ which improves the right hand side inclusion in (\ref{Sc}).
\item[2)]From Corollary \ref{increasingintheta}, the family of sets $\displaystyle \frac{K+_\theta L}{1-\theta^\frac{1}{n}}$ is increasing with respect to $\theta$, and using the equivalence $1-\theta\sim n(1-\theta^\frac{1}{n})$, the existence of the limiting convolution set with $\alpha=1$,
    $$
    C_1(K,L)=\lim_{\theta\to 1^-}\frac{K+_\theta L}{1-\theta}
    $$
    follows.
However, there are cases in which this set is unbounded. We refer to Example 3.15 in \cite{Tso}
for a detailed construction of an example
 where the limiting convolution is $\Ren$. In that paper
sufficient conditions for $C_1(K,L)$ to be bounded are also given.
\item[3)]The righten set in Proposition \ref{Schmuckeslagerseveralbodies} is $n(1-\theta^\frac{1}{n})C_1(K,L)$. The previous result can be deduced from Corollary \ref{increasingintheta} letting $\theta_0\rightarrow 1^{-}$ (see the proof of Theorem \ref{Zhangextension} below).
\item[4)]Nevertheless, a general inclusion $\displaystyle{K+_\theta L\subseteq n(1-\theta^\frac{1}{n})C}$, for some body $C$ independent from $\theta$ can not be proved, since it was shown in \cite{Tso} that the limiting convolution body with $\alpha=1$ could be non compact.
\end{rmk}

The left-hand side inclusion in (\ref{Sc}) is generalized with the following Proposition. Recall that $\conv(A)$ denotes the convex hull of a set $A$.

\begin{prop}\label{leftinclusion}
Let $K,\ L\subset \Re^n$ be convex bodies satisfying (\ref{Mzero}). Then, for every $\theta\in(0,1)$
$$K+_\theta L\supseteq(1-\theta)M(K,L)\conv(\Pi^*K\cup\Pi^*L).$$
\end{prop}

\begin{proof}
In \cite{M-R-S} it is proved that, given two convex bodies $K$, $L$, and $u\in S^{n-1}$, the function $f(r)=|K\cap(ru+L)|$ verifies $f^{'}(0)=|C_u^+(1,2)|-|C_u^-(2,1)|,$ where
$$C_u^+(1,2)=P_u(K\cap L)\cap\{\psi_K^+>\psi_L^+\geq \psi_K^->\psi_L^-\}$$
$$C_u^+(2,1)=P_u(K\cap L)\cap\{\psi_L^+\geq\psi_K^+> \psi_L^-\geq\psi_K^-\}$$
with $$\psi_K^+(y)=\max\{t\ :\ tu+y\in K\}$$ and
$$\psi_K^-(y)=\min\{t\ :\ tu+y\in K\}.$$

Thus $\frac{d}{d\overrightarrow{v}}|K\cap (x-L)|\geq -|P_v^+ K\cap(x-L)|\geq-\min\{|P_v^+K|,|P_v^+L|\}$. Consequently,
\begin{align*}
|&K\cap(x-L)|=\\
            &M(K,L)+\int_0^{|x|}\frac{d}{dt}\left|K\cap \left(t\frac{x}{|x|}-L\right)\right|\geq M(K,L)-|x|\min\{|P_v^+K|,|P_v^+L|\}.
\end{align*}
Hence, if $M(K,L)-|x|\min\{|P_v^+K|,|P_v^+L|\}\geq \theta M(K,L)$ then $x\in K+_\theta L$ and this holds if and only if
$\min\{||x||_{\Pi^*K},||x||_{\Pi^*L}\}\leq (1-\theta)M(K,L)$.

So, $(1-\theta)M(K,L)(\Pi^*K\cup \Pi^*L)\subseteq K+_\theta L$. The convexity of the set $K+_\theta L$ yields the desired result.
\end{proof}

\begin{rmk}
Taking $L=-K$ and $|K|=1$, we recover the left hand side inclusion in (\ref{Sc}).
\end{rmk}
\begin{rmk}
Applying Zhang inequality we deduce that 
$$
\min\{|K|^{n-1},|L|^{n-1}\}\left|\frac{C_1(K,L)}{M(K,L)}\right|\geq{2n \choose n}\frac{1}{n^n}
$$
which extends Zhang inequality (\ref{inq:zhang2}). Nevertheless, a stronger extension of Zhang inequality can be proved using Corollary \ref{increasingintheta}.
\end{rmk}

\begin{thm}\label{Zhangextension}
Let $K,L\subset\Re^n$ be convex bodies such that $\displaystyle{M(K,L)=|K\cap (-L)|}$. Then
\begin{equation}
\label{zhang:ext}
|C_1(K,L)|\geq{2n\choose n}\frac{1}{n^n}\frac{|K||L|}{M(K,L)}
\end{equation}
Equality holds if and only if $K=-L$ is a simplex.
\end{thm}

\begin{proof}
From Corollary \ref{increasingintheta} we have that for every $0\leq\theta_0\leq \theta<1$
\begin{equation}\label{zhang:ext:inclusion}
\frac{K+_{\theta_0} L}{1-\theta_0^\frac{1}{n}}\subseteq\frac{K+_{\theta} L}{1-\theta^\frac{1}{n}}.
\end{equation}
Thus, letting $\theta\to 1^-$ we obtain that for every $\theta_0\in[0,1)$
$$
\frac{K+_{\theta_0} L}{1-\theta_0^\frac{1}{n}}
\subseteq
\lim_{\theta\to 1^-}
\frac{1-\theta}{1-\theta^\frac{1}{n}}\frac{K+_{\theta}L}{1-\theta}=nC_1(K,L),
$$
and taking volumes
$$
|K+_{\theta_0} L|\leq n^n(1-\theta_0^\frac{1}{n})^n|C_1(K,L)|
$$
for $\theta_0\in[0,1)$. Integrating over $[0,1]$ yields
$$
\int_0^1|K+_{\theta_0} L|\,d\theta_0
\leq
n^n|C_1(K,L)|\int_0^1 (1-\theta_0^\frac{1}{n})^n\,d\theta_0=
n^n|C_1(K,L)|{2n\choose n}^{-1}.
$$
Integrating by parts and using Fubini's Theorem, the first integral equals
\begin{align*}
\int_0^1|K+_{\theta_0} L|\,d\theta_0
&=
\int_0^1\left|\left\{x\in K+L:\frac{|K\cap (x-L)|}{M(K,L)}\ge \theta_0 \right\}\right|\,d\theta_0
\\
&=
\int_{K+L}\frac{|K\cap (x-L)|}{M(K,L)}\,dx
=\frac{|K| |L|}{M(K,L)}
\end{align*}
from which the desired inequality follows.

If equality holds in (\ref{zhang:ext}), then (\ref{zhang:ext:inclusion}) holds also with equality for any $0\leq\theta_0\leq \theta<1$. Letting $\theta_0=0$,
\begin{equation}\label{zhang:ext:equality}
({1-\theta^\frac{1}{n}})(K+ L)={K+_{\theta} L}
\end{equation}
for any $\theta\in[0,1)$. Now Corollary \ref{inclusion:BM} implies
$$
\theta^{\frac{1}{n}}(K+_1L)
+
(1-\theta^{\frac{1}{n}})(K+L)
\subseteq
K+_\theta L
=
(1-\theta^{\frac{1}{n}})
(K+L).
$$
A compactness argument shows that $K+_1 L=\{0\}$, so equality (\ref{zhang:ext:equality}) holds for every $0\le\theta\le1$, and then (ii) in  Proposition \ref{equalityconvexity2} implies that $K=-L$ is a simplex.
\end{proof}

Finally, Corollary \ref{increasingintheta} allows us to recover Rogers-Shephard inequality (\ref{eqn:RoS3}). We also solve the problem of characterizing equality cases posed in \cite{Ro-Sh2}.

\begin{prop}\label{rmk:RoS3}
Let $K,L\subset\Re^n$ be convex bodies. Then
$$
|K+L|M(K,L)\leq {2n \choose n}|K||L|.
$$
Equality holds if and only if $K=-L$ is a simplex.
\end{prop}

\begin{proof}
By a translation we may assume that (\ref{Mzero}) is satisfied. A similar argument to that used in Theorem \ref{Zhangextension}, taking $\theta_0=0$ in Corollary \ref{increasingintheta}, applying volumes and integrating in $\theta$, shows the desired result.

If equality holds, then (ii) in Proposition \ref{equalityconvexity2} is satisfied, and therefore  $K=-L$ is a simplex.
\end{proof}

\section{Convolution of $m$ bodies}
\label{mbodies}

In this section we will extend the definition of $\theta$-convolution bodies to more than two sets. The $\theta$-convolution is not associative (as a simple computation with Euclidean balls of different radius shows) so a  definition of an $m$-fold convolution can not be made inductively. Nevertheless, since $|K\cap \left(x-L\right)|=\chi_{K}\ast\chi_L(x)$ and the convolution is associative, it seems natural to make the following extension of $\theta$-convolution bodies:

\begin{defn}
Let $\mcuerpos$ be $m$ measurable sets in $\Re^n$ and let $\theta\in [0,1]$. We define their $\theta$-convolution as the set
$$
K_1+_\theta\dots+_\theta K_m=\{x\in K_1+\cdots+K_m\,:\,\chi_{K_1}\ast\dots\ast\chi_{K_m}(x)\geq \theta M(K_1,\dots,K_m)\}
$$
when $M(\mcuerpos)=\max_{x\in\Re^n}\chi_{K_1}\ast\dots\ast\chi_{K_m}(x)$ is finite.
\end{defn}

For $\theta=0$ the set $K_1+_0\dots+_0K_m$ is just the support of the function $\chi_{K_1}\ast\dots\ast\chi_{K_m}$, the usual Minkowski sum $K_1+\dots+K_m$.

The commutative and associative properties of the convolution imply trivially that
\begin{enumerate}
    \item[(a1)] $\chi_{K_1}\ast\dots\ast\chi_{K_m}=\chi_{K_{\sigma(1)}}\ast\dots\ast\chi_{K_{\sigma(m)}}$ for any $\sigma $ permutation of $\{1,...,m\}$.
    \item[(a2)] $\chi_{z+K_1}\ast\dots\ast\chi_{K_m}(x)=\chi_{K_1}\ast\dots\ast\chi_{K_m}(x-z)$
    \item[(a3)]
            $\chi_{TK_1}\ast\dots\ast\chi_{TK_m}(x)=|\det T|^{m-1}
            \chi_{K_1}\ast\dots\ast\chi_{K_m}(T^{-1}x)$ for any $T\in GL_n(\Re)$.
\end{enumerate}

Consequently we have the following result, analogous to Proposition \ref{prop:properties}.

\begin{prop}
Let $\mcuerpos$ be compact sets in $\Re^n$, $\lambda\in\Re, \theta\in [0,1]$, $x\in\Re^n$  and $T\in GL_n(\Re)$. Then:

\begin{enumerate}
\item[(b1)] $(\lambda K_1)+_{\theta}\dots+_\theta(\lambda K_m)=\lambda(K_1+_{\theta}\dots+_\theta K_m)$
\item[(b2)] $K_1+_{\theta}\dots+_\theta K_m=K_{\sigma(1)}+_{\theta}\dots+_\theta K_{\sigma(m)}$ for any $\sigma$ permutation of $\{1,...,m\}$.
\item[(b3)] $(x+K_1)+_\theta K_2+_{\theta}\dots+_\theta K_m=x+(K_1+_{\theta}\dots+_\theta K_m)$
\item[(b4)] $TK_1+_{\theta}\dots+_\theta TK_m=T(K_1+_{\theta}\dots+_\theta K_m)$
\end{enumerate}

\end{prop}

The convexity is transmitted to the $\theta$-convolution of $m$ convex bodies.

\begin{prop}
Let $\mcuerpos$ be convex bodies in $\Re^n$. Then $K_1+_{\theta}\dots+_\theta K_m$ is a convex body.
\end{prop}

\begin{proof}
The characteristic function of each convex body $K_i$ is log-concave. The convolution of log-concave functions is log-concave, and the level sets of log-concave functions are convex.
\end{proof}

Corollary \ref{increasingintheta} is proved by using that $\chi_{K_1}\ast\chi_{K_2}$ is $\frac{1}{n}$-concave in its support. In order to generalize this result, we have to prove that  the convolution of more than two characteristic functions is $s$-concave for some s. We get this result for $s^{-1}=(m-1)n$ by considering sets in dimension $(m-1)n$. As in the case $m=2$, we may assume, without loss of generality, that
\begin{equation}\label{Mmzero}
M(K_1,\dots,K_m)=\chi_{K_1}\ast\dots\ast\chi_{K_m}(0).
\end{equation}

\begin{prop}\label{incconvexitym}
Let $\mcuerpos\subset\Ren$ be convex bodies satisfying (\ref{Mmzero}). Then, for any $\theta_1,\theta_2, \lambda_1,\lambda_2 \in[0,1]$ such that $\lambda_1+\lambda_2\leq 1$ we have
    \begin{equation}
    \label{inclusion1:convexitym}
    \lambda_1 (K_1+_{\theta_1}\dots+_{\theta_1}K_m)+ \lambda_2(K_1+_{\theta_2}\dots+_{\theta_2}K_m)
    \subseteq
    K_1+_{\theta}\dots+_{\theta}K_m,
    \end{equation}
    where $1-\theta^{\frac{1}{(m-1)n}}=\lambda_1(1-\theta_1^{\frac{1}{(m-1)n}})+\lambda_2(1-\theta_2^{\frac{1}{(m-1)n}}).$
\end{prop}

\begin{proof}
First at all, notice that for any $x\in\Ren$,
\begin{align*}
&\chi_{K_1}\ast\dots\ast\chi_{K_m}(x)
\\
&=\int_{\Ren} \chi_{K_1}\ast\dots\ast\chi_{K_{m-1}}(t_{m-1})\chi_{K_m}(x-t_{m-1})\,dt_{m-1}
\\
&=\int_{\Ren}\int_{\Ren} \chi_{K_1}\ast\dots
\\
&\hskip1cm\dots\ast\chi_{K_{m-2}}(t_{m-2})
\chi_{K_{m-1}}(t_{m-1}-t_{m-2})\chi_{K_{m}}(x-t_{m-1})\,dt_{m-2}\,dt_{m-1}
\\
&=\cdots
\\
&=\int_{\Ren}\cdots\int_{\Ren}
\chi_{K_1}(t_1)
\chi_{K_{2}}(t_{2}-t_{1})\dots
\\
&\hskip1cm\dots
\chi_{K_{m-1}}(t_{m-1}-t_{m-2})
\chi_{K_{m}}(x-t_{m-1})\,dt_1\dots dt_{m-1}
\\
&=|\Omega_{m-1}(x)|
\end{align*}
where
\begin{align*}
\Omega_{m-1}(x)=\{(t_1,\dots t_{m-1})\in\Re^{(m-1)n}:\,
&t_1\in K_1,\,
t_2-t_1\in K_2,\,
\dots
\\
\dots\,
&t_{m-1}-t_{m-2}\in K_{m-1},\,
x-t_{m-1}\in K_m
\}.
\end{align*}
The convexity of $\mcuerpos$ gives
\begin{equation}
\label{inclusion0:convexitym}
\Omega_{m-1}(\lambda_1x_1+\lambda_2x_2)
\supseteq
(1-\lambda_1-\lambda_2)\Omega_{m-1}(0)
+
\lambda_1 \Omega_{m-1}(x_1)
+
\lambda_2 \Omega_{m-1}(x_2)
\end{equation}
for any $x_1,x_2\in\Ren$ and $\lambda_1,\lambda_2\ge0$ such that $\lambda_1+\lambda_2\le1$. In particular, if $x_1\in K_1+_{\theta_1}\cdots +_{\theta_1}K_m$ and $x_2\in K_1+_{\theta_2}\cdots +_{\theta_2}K_m$, Brunn-Minkowski inequality in $\Re^{(m-1)n}$ implies
\begin{equation}
\label{BM:convexitym}
|\Omega_{m-1}(\lambda_1x_1+\lambda_2x_2)|
\ge
((1-\lambda_1-\lambda_2)+\lambda_1 \theta_1^\frac{1}{(m-1)n}+\lambda_2\theta_2^\frac{1}{(m-1)n})^{(m-1)n}
|\Omega_{m-1}(0)|,
\end{equation}
which shows that $\lambda_1x_1+\lambda_2x_2\in K_1+_{\theta}\dots+_{\theta}K_m$ with
$$
1-\theta^{\frac{1}{(m-1)n}}=\lambda_1(1-\theta_1^{\frac{1}{(m-1)n}})+\lambda_2(1-\theta_2^{\frac{1}{(m-1)n}}).
\vspace{-.9cm}
$$
\end{proof}

\begin{rmk}\label{boundarym}
The concavity of the function
$$
x\in \sumam{}\mapsto \chi_{K_1}\ast\dots\ast\chi_{K_m}(x)^{\frac{1}{(m-1)n}}
$$
allows us to write, for any $\theta\in[0,1)$, the boundary of $K_1+_\theta\cdots+_\theta K_m$ as
$$
\partial (\sumam{\theta})
=
\{x\in \sumam{}: \chi_{K_1}\ast\dots\ast\chi_{K_m}(x)=\theta M(\mcuerpos)\}.
$$
In particular, for $\theta=0$,
$$
\partial (\sumam{})
=
\{x\in \sumam{} : \chi_{K_1}\ast\dots\ast\chi_{K_m}(x)=0\}.
$$
That implies that for any $x\in \sumam{}$, $x\not\in \sumam{1}$, there exists a (unique) $\theta\in [0,1)$ such that $x\in\partial (\sumam{\theta})$.
\end{rmk}

\begin{rmk}\label{convolution1:equalitym}
Taking $\theta_1=\theta_2=1$, and following the proof of Proposition \ref{incconvexitym}, we get equality in the inclusion (\ref{inclusion0:convexitym}) and in the Brunn-Minkowski inequality (\ref{BM:convexitym}). Consequently, $\Omega_{m-1}(x)$ are all homothetic to $\Omega_{m-1}(0)$ for any $x\in K_1+_1\cdots+_1 K_m$. Using that they all have the same volume, we get that they are translations of $\Omega_{m-1}(0)$.
\end{rmk}

In particular, taking $\theta_1=\theta_2$ and $\lambda_2=1-\lambda_1$ in Proposition \ref{incconvexitym}, we get the convexity of $K_1+_{\theta}\dots+_{\theta}K_m$. Also, taking $\theta_1=\theta_2=\theta_0$, we get the version of Corollary \ref{increasingintheta} to $m$ bodies.

\begin{cor}\label{increasinginthetam}
Let $\mcuerpos\subset\Re^n$ be convex bodies such that (\ref{Mmzero}) is satisfied. Then for any $0<\theta_0\le\theta<1$
$$
\frac{\sumam{\theta_0}}{1-\theta_0^\frac{1}{(m-1)n}}
\subseteq
\frac{\sumam{\theta}}{1-\theta^\frac{1}{(m-1)n}}.
$$
\end{cor}

We can easily get the Brunn-Minkowski type inequality for $m$ bodies, a generalization of Corollary \ref{BMtheta}.

\begin{cor}
Let $\mcuerpos\subset\Re^n$ be convex bodies. Then
\begin{equation*}
|\sumam{\theta}|^{\frac{1}{n}}\geq \left(1-\theta^{\frac{1}{(m-1)n}}\right)(|K_1|^{\frac{1}{n}}+\cdots +|K_m|^{\frac{1}{n}}).
\end{equation*}

\end{cor}

Finally, we get Zhang and Roger-Shephard type inequalities for $m$ bodies. For $\mcuerpos\subseteq\Ren$, let
$$
C_1(\mcuerpos)
=\lim_{\theta\to 1}\frac{\sumam{\theta}}{1-\theta}.
$$
As in the case $m=2$, the existence of the previous limit follows from Corollary \ref{increasinginthetam}.

\begin{cor}
Let $\mcuerpos\subseteq\Ren$ be convex bodies. Then
$$
|\sumam{}|
\le
{m\,n\choose n}\frac{|K_1|\cdots |K_m|}{M(\mcuerpos)}
\le
(m-1)^n n^n
|C_1(\mcuerpos)|
$$
\end{cor}

\begin{proof}
The proof runs as in Theorem \ref{Zhangextension} and Proposition \ref{rmk:RoS3} resp., using  Corollary \ref{increasinginthetam}, instead of Corollary \ref{increasingintheta}.
\end{proof}

Regarding the study of equality cases, we will show that equality never occur in  (\ref{inclusion1:convexitym}) for all $\lambda_1,\lambda_2$ such that $\lambda_1+\lambda_2=1$  unless $m=2$ or $n=1$. That implies that extensions of Zhang, Roger-Shephard and Brunn-Minkowski type inequalities are not sharp for $m>2$ and $n>1$.

The following result can be proved as in Proposition \ref{equalityconvexity1}.

\begin{prop}\label{equalitycases:convexitym}
Let $\mcuerpos\subset\Ren$ be convex bodies satisfying (\ref{Mmzero}) and such that for any $\theta_1,\theta_2, \lambda_1,\lambda_2 \in[0,1]$ such that $\lambda_1+\lambda_2= 1$ we have
    \begin{equation}
    \label{equality:convexitym}
    \lambda_1 (\sumam{\theta_1})+ \lambda_2(\sumam{\theta_2})
    =
    \sumam{\theta},
    \end{equation}
    where $\theta^{\frac{1}{(m-1)n}}=\lambda_1\theta_1^{\frac{1}{(m-1)n}}+\lambda_2\theta_2^{\frac{1}{(m-1)n}}.$ Then for every $x\in \sumam{}$ $\Omega_{m-1}(x)$ is homothetic to $\Omega_{m-1}(0)$ .
\end{prop}

Then we will show that this consequence can not occur for $m\ge3$.

\begin{prop}\label{impossiblem}
Let $m\ge3$ and $\mcuerpos\subset\Ren$ be convex bodies satisfying (\ref{Mmzero}). Then it is not possible that $\Omega_{m-1}(x)$ is homothetic to $\Omega_{m-1}(0)$ for every $x\in \sumam{}$.
\end{prop}

For the proof, we will use the following fact
on sum of simplices, which is of independent interest.

\begin{lem}\label{sum_simplices1}
Let $K,L\subset\Ren$ be $n$-dimensional  convex bodies. If $K+L$ is an $n$-dimensional simplex, then $K$ and $L$ are both $n$-dimensional  homothetic simplices.
\end{lem}

\begin{proof} Denote by $h_C(x)=\max\{\langle x,y\rangle: y\in C\}$ the support function of a compact set $C\subset\Ren$.

Let us write $K+L=\conv\{w_0,\dots w_n\}$. We will show that $L$ is a simplex.

Each $w_i$ is an extreme point of $K+L$. Then $w_i=u_i+v_i$ where $u_i\in K$ and $v_i\in L$ are extreme points of $K$ and $L$ resp. In particular, $v_i=w_i-u_i\in L$ and so $\conv\{w_0-u_0,\dots w_n-u_n\}\subseteq L$. Also
\begin{align*}
h_K(x)+h_L(x)&=h_{K+L}(x)=\max_{0\le i \le n}\langle x,w_i\rangle
\\
&
=
\max_{0\le i \le n}\left(\langle x,u_i\rangle+\langle x,w_i-u_i\rangle\right)
\\
&\le
\max_{0\le i \le n}\langle x,u_i\rangle+\max_{0\le i \le n}\langle x,w_i-u_i\rangle
\\
&\le
h_K(x)+\max_{0\le i \le n}\langle x,w_i-u_i\rangle
\end{align*}
Thus
$
h_L(x)\le \displaystyle\max_{0\le i \le n}\langle x,w_i-u_i\rangle,
$
so $L\subseteq \conv\{w_0-u_0,\dots w_n-u_n\}$, and $L$ is a simplex.

Now, write $K=\conv\{u_0,\dots, u_n\}$ and $L=\conv\{v_0,\dots v_n\}$. We will prove that they are homothetic simplices.

Let $A_i=\{s\in S^{n-1}: h_K(s)=\langle x,u_i\rangle\}$ and $B_i=\{s\in S^{n-1}: h_L(s)=\langle x,v_i\rangle\}$. $K$ and $L$ are homothetic if and only if both partitions of $S^{n-1}$
$$
S^{n-1}=\bigcup_{i=0}^n A_i=\bigcup_{i=0}^n B_i
 $$
are identical. Assume they are not the same, then the partition $S^{n-1}=\displaystyle\bigcup_{i,j=0}^nC_{i,j}$ where $C_{i,j}=\{s\in S^{n-1}: h_{K+L}(s)=\langle x,u_i+v_j\rangle\}$ has more than $n+1$ elements, since
$$
h_{K+L}(s)=h_K(s)+h_L(s)
=
\max_{0\le i \le n}\langle x,u_i\rangle
+
\max_{0\le j \le n}\langle x,v_j\rangle
=
\max_{0\le i,j \le n}\langle x,u_i+v_j\rangle.
$$
Then $K+L$ is not a simplex.
\end{proof}

\begin{rmk}\label{sum_simplices2}
An argument similar to that used in the remark before Theorem 3.2.3 in \cite{Schn1} leads to the same claim. However, we have included a direct proof of it for the sake of completeness.
%
\end{rmk}

\begin{proof}[Proof of Proposition \ref{impossiblem}]
Notice that
\begin{align*}
\Omega_{m-1}(x)=\{(t_1,\dots t_{m-1})\in\Re^{(m-1)n}:\,
&(t_1,\dots t_{m-2})\in \Omega_{m-2}(t_{m-1}),
\\
&t_{m-1}\in  (x-K_m)\cap(K_1+\cdots+K_{m-1})
\}.
\end{align*}
and that $\Omega_{m-2}(t_{m-1})$ is non-empty  if and only if $t_{m-1}\in K_1+\cdots K_{m-1}$. Then the projection onto the $t_{m-1}$ coordinate is
$$
P_{m-1}(\Omega_{m-1}(x))=(x-K_m)\cap(K_1+\cdots+K_{m-1}).
$$

Suppose that $\Omega_{m-1}(x)$ is homothetic to $\Omega_{m-1}(0)$ for every $x\in \sumam{}$. Then their projections are also homothetic. Then $(x-K_m)\cap(K_1+\cdots+K_{m-1})$ is homothetic to $(-K_m)\cap(K_1+\cdots+K_{m-1})$ for any $x\in \sumam{}$ and Soltan's Theorem \ref{soltanth} implies that $K_1+\cdots K_{m-1}$ and $-K_m$ are homothetic simplices.

On the other hand, the projection onto the $t_{m-2}$ coordinate is
\begin{align*}
P_{m-2}(\Omega_{m-1}(x))
&=
\bigcup_{t_{m-1}\in x-K_m}P_{m-2}\left(\Omega_{m-2}(t_{m-1})\right)
\\&=
\bigcup_{t_{m-1}\in x-K_m}(t_{m-1}-K_{m-1})\cap(K_1+\cdots+K_{m-2})
\\&=
(x-(K_{m-1}+K_m))\cap(K_1+\cdots+K_{m-2})
\end{align*}
and so we have that for every $x\in\sumam{}$, $(x-(K_{m-1}+K_m))\cap(K_1+\cdots+K_{m-2})$ is homothetic to $(-(K_{m-1}+K_m))\cap(K_1+\cdots+K_{m-2})$. Soltan's Theorem \ref{soltanth} shows again that
$K_1+\cdots+K_{m-2}$ and $-(K_{m-1}+K_m)$ are homothetic simplices.

If $K_m$ and $K_{m-1}+K_m$ are both simplices, by Lemma \ref{sum_simplices1}, $K_{m-1}$ is also a simplex homothetic to $K_m$. But
$
(K_1+\cdots+K_{m-1})+K_m
$
is a simplex, so Lemma \ref{sum_simplices1} shows that $K_1+\cdots+K_{m-1}$ and $K_m$ are homothetic. But then $K_m$ is a simplex homothetic to $-K_m$, a contradiction, unless $n=1$.
\end{proof}

 \section*{Acknowledgements}
Part of this work was done while the first two named authors were attending the Thematic Program on Asymptotic Geometric Analysis in Fall 2010 at the Fields Institute in Toronto, where the first named author enjoyed the J. Marsden's postdoctoral fellowship. We appreciate the hospitality. We are indebted to Prof. A. Giannopoulos for many helpful discussions and for pointing out to us several references at the early stages of this work. We also would like to thank Steven Taschuk for pointing out to us the alternative proof of Lemma \ref{sum_simplices1} exposed in Remark \ref{sum_simplices2}. Finally, the authors thank the referee for suggesting the study of equality cases and other improvements in exposition.

\bibliography{Thesis}
\bibliographystyle{ieeetr}

\end{document}